\documentclass{amsart}

\newtheorem{theorem}{Theorem}[section]
\newtheorem{lemma}[theorem]{Lemma}

\theoremstyle{definition}
\newtheorem{definition}[theorem]{Definition}

\newtheorem{proposition}[theorem]{Proposition}
\newtheorem{corollary}[theorem]{Corollary}

\theoremstyle{remark}

\numberwithin{equation}{section}

\begin{document}

\title[the porous medium equation with  potential under  geometric flow]{ Harnack estimates  for  the porous medium equation with  potential under  geometric flow }

\author{Shahroud Azami}
\address{Department of Mathematics, Faculty of Sciences
Imam Khomeini International University,
Qazvin, Iran. }

\email{azami@sci.ikiu.ac.ir}






\begin{abstract}
Let $(M, g(t))$, $t\in[0,T)$ be a closed Riemannian $n$-manifold whose Riemannian metric  $g(t)$ evolves by the geometric  flow
$
\frac{\partial }{\partial t} g_{ij}=-2S_{ij}
$,
where $S_{ij}(t)$  is a symmetric  two-tensor  on $(M,g(t))$. We discuss differential Harnack estimates for  positive solution to the porous medium equation with  potential, $\frac{\partial u}{\partial t}=\Delta u^{p}+S u$, where $S=g^{ij}S_{ij}$ is the  trace of $S_{ij}$, on time-dependent Riemannian metric evolving by the above geometric flow.
\end{abstract}

\maketitle
\section{Introduction}
There are many  results about the  Harnack estimates for parabolic equations. The study  of differential  Harnack estimates  and applications for parabolic equation  originated in the famous paper \cite{PLY} of Li and Yau, in which  they discoverd the celebrated differential Harnack estimate for any positive  solution  to the heat equation with potential on Riemannian manifolds with a fixed Riemannian metric. After then, this method  plays an important role  in the study  of geometric flows, for instance, Hamilton  proved  Harnack inequalities for the  Ricci flow on Riemannian manifolds with weakly positive curvature operator \cite{H1} and mean curvature flow \cite{H2}, also see \cite{HD, BC}. Also, recently many authors obtained  a differential Harnack estimate for  solutions of the parbolic  equation on Riemannian manifold along the geometric flow, for instance, Fang in \cite{SF}, proved  differential Harnack estimates for  backward heat equation with potentials under an extended Ricci flow and Ishida in \cite{MI} studied  differential Harnack estimates for   heat equation with potentials along  the geometric flow.

Let $M$  be a closed Riemannian manifold with a one parameter family of Riemannian  metric   $g(t)$ evolving  by the geometric  flow
\begin{equation}\label{1}
\frac{\partial }{\partial t} g_{ij}(x,t)=-2S_{ij}(x,t)
\end{equation}
where $S_{ij}$  is a general  time-dependent symmetric  two-tensor  on $(M,g(t))$. For example, (\ref{1}) becomes Ricci flow whenever  $S_{ij}=R_{ij}$ is the Ricci tensor, where it introduced by Hamilton \cite{H3}.

In \cite{CZ}, Cao and Zhu obtained Aronson-B\'{e}nilan estimates  for  the porous medium equation  (PME) with potential
\begin{equation}\label{5}
\frac{\partial u}{\partial t} =\Delta u^{p}+R u
\end{equation}
along the  Ricci flow, where  $R$ is the scalar curvature of $M$. Differential equations (\ref{5}) is a nonlinear parabolic equation and has applications in mathematics and physics.  For $p>1$ differential  equations  PME describes physical processes of gas through porous medium, heat radiation in plasmas (\cite{LV}). Motivated by the above works, in this paper, we consider equation of type (\ref{5}) with a linear forcing term
 \begin{equation}\label{6}
\frac{\partial u}{\partial t} =\Delta u^{p}+Su
\end{equation}
under the geometric flow ( \ref{1}), where  $S=g^{ij}S_{ij}$, $\Delta $ is Laplace operator with n respect to the evolving metric $g(t)$ of the geometric flow ( \ref{1}) and prove  differential Harnack estimates for  positive solutions to  (\ref{6}).
Notice also that for any smooth solution  $u$  of (\ref{6})  we have
$$\frac{\partial }{\partial t}(\int_{M}u\,d\mu)=\int_{M}\frac{\partial u}{\partial t}d\mu+u\frac{\partial d\mu}{\partial t}=\int_{M}(\frac{\partial u}{\partial t}-Su)d\mu =\int_{M} \Delta u^{p}d\mu=0.$$
For $p=1$, (\ref{6}) is simply the equation
 \begin{equation}\label{6a}
\frac{\partial u}{\partial t} =\Delta u+Su,
\end{equation}
where differential Harnack estimates for positive solution  to (\ref{6a}) have been studied in  \cite{MI}.
 Suppose that $u$ is positive solution of  (\ref{6})  and $v=\frac{p}{p-1}u^{p-1}$. Then  we can rewrite (\ref{6})  as follows
\begin{equation}\label{7}
\frac{\partial v}{\partial t} =(p-1)v\Delta v+|\nabla v|^{2}+(p-1)Sv.
\end{equation}
To state the main results of the current article, analogous to definition from M\"{u}ller  (\cite{M}) we introduce  evolving  tensor quantises associated with  the tensor  $S_{ij}$.

\begin{definition}
 Let $g(t)$  be a solution  of the geometric flow (\ref{1})  and let $X=X^{i}\frac{\partial }{\partial x^{i}}\in\mathcal{X}(M)$ be  a vector field  on $(M, g(t))$. We define
\begin{eqnarray}\label{2}
\mathcal{I}(S,X)&=&(R^{ij}-S^{ij})X_{i}X_{j},\\\label{3}
\mathcal{H}(S,X)&=&\frac{\partial S}{\partial t}+\frac{S}{t}-2\nabla_{i}SX^{i}+2S^{ij}X_{i}X_{j},\\
\label{4}
\mathcal{D}(S)&=&\frac{\partial S}{\partial t}-\Delta S-2|S_{ij}|^{2},\\
\mathcal{E}(S,X)&=&\mathcal{D}(S)+2\mathcal{I}(S,X)+2(2\nabla^{i}S_{ij}-\nabla_{j}S)X^{j}.
\end{eqnarray}
\end{definition}
\section{Main results}
The main results of this paper are the following.
\begin{theorem}\label{t1}
Let $g(t)$, $t\in[0,T)$ be a solution to the geometric flow (\ref{1}) on a closed   Riemannian $n$-manifold $M$ satisfying
\begin{equation}
\mathcal{E}(S,X)\geq0,\,\,\,\mathcal{H}(S,X)\geq0,\,\,\,Ric\geq -(n-1)k_{1},\,\,\,\,-k_{2}g\leq S_{ij}\leq k_{3}g,\,\,\,\,\,S\geq0
\end{equation}
for all vector fields $X$ and  all time $t\in[0,T)$. Suppose $u$ is  a smooth positive  solution  to equation (\ref{6}) with $p>1$  and $v=\frac{p}{p-1}u^{p-1}$. Then
for any $d\in[2,\infty)$, on the geodesic ball $\mathcal{Q}_{\rho,T}$, we have
\begin{equation}\label{ft1}
\frac{|\nabla v|^{2}}{v}-2\frac{v_{t}}{v}-\frac{S}{v}-\frac{d}{t}\leq \frac{2n(p-1)}{1+n(p-1)}(\frac{E_{1}v_{\max}}{\rho^{2}}+E_{2})
\end{equation}
where
$E_{1}=\big(p^{2}n+\frac{1}{2}\sqrt{k_{1}}\rho+\frac{9}{4}\big)c_{1}(p-1)$,
$E_{2}=\sqrt{c_{2}}(k_{2}+k_{3})^{2}+1$ and $c_{1},c_{2}$ are absolute positive  constants.
\end{theorem}

Let $\rho\to\infty$, we can  get the gradient estimates for the nonlinear parabolic equation (\ref{6}).
\begin{corollary}\label{c1}
Let $g(t)$, $t\in[0,T)$ be a solution to the geometric flow (\ref{1}) on a closed   Riemannian $n$-manifold $M$ satisfying
\begin{equation}
\mathcal{E}(S,X)\geq0,\,\,\,\mathcal{H}(S,X)\geq0,\,\,\,Ric\geq -(n-1)k_{1},\,\,\,\,-k_{2}g\leq S_{ij}\leq k_{3}g,\,\,\,\,\,S\geq0
\end{equation}
for all vector fields $X$ and  all time $t\in[0,T)$. Suppose $u$ is  a bounded  smooth positive  solution  to equation (\ref{6}) with $p>1$  and $v=\frac{p}{p-1}u^{p-1}$. Then
for any $d\in[2,\infty)$, on the geodesic ball $\mathcal{Q}_{\rho,T}$, we have
\begin{equation}\label{ft2}
\frac{|\nabla v|^{2}}{v}-2\frac{v_{t}}{v}-\frac{S}{v}-\frac{d}{t}\leq \frac{2n(p-1)}{1+n(p-1)}E_{2}
\end{equation}
where
$E_{2}=\sqrt{c_{2}}(k_{2}+k_{3})^{2}+1$ and $c_{2}$ is absolute positive  constant.

\end{corollary}

As an application, we get the following Harnack inequality for $v$.
\begin{theorem}\label{t2}
With the same assumption  as in Corollary \ref{c1}, if $d\geq 2$, then  for any points $(x_{1},t_{1})$ and $(x_{2}, t_{2})$ on $M\times [0,T)$ with $0<t_{1}<t_{2}$ we have the following estimate
\begin{equation}\label{ftt1}
v(x_{1},t_{1})\leq v(x_{2},t_{2})(\frac{t_{2}}{t_{1}})^{\frac{d}{2}}exp\big(\frac{\Gamma}{2v_{\min}}+(\frac{n(p-1)}{1+n(p-1)}E_{2})(t_{2}-t_{1}) \big)
\end{equation}
where $E_{2}$ is the  constants  in Corollary \ref{c1} and $\Gamma=\inf_{\gamma}\int_{t_{1}}^{t_{2}}(S+|\frac{d\gamma}{dt}|^{2})dt$ with the infimum taking over all smooth curves $\gamma(t)$ in $M$, $t\in[t_{1},t_{2}]$, so that $\gamma(t_{1})=x_{1}$ and $\gamma(t_{2})=x_{2}$.
\end{theorem}

Our results in this article are similar to those of Cao and Zhu \cite{CZ} in the case $S_{ij}=R_{ij}$.
\section{Examples}
\subsection{Static Riemannian manifold}
In this case we have $S_{ij}=0$ and $S=0$. Then $\mathcal{D}=0$,  $\mathcal{H}(S,X)=0$ and  $\mathcal{I}(S,X)=R^{ij}X_{i}X_{j}$. Thus the assumption  in Theorems \ref{t1}, \ref{t2} and Corollary \ref{c1} can be replace by $R_{ij}\geq 0$.
\subsection{The Ricci flow}
The Ricci flow defined for the first time  by Haimlton as follow
\begin{equation*}
\frac{\partial }{\partial t}g_{ij}=-2R_{ij}.
\end{equation*}
In this case we get $S_{ij}=R_{ij}$ and $S=R$ the scalar curvature. Along the Ricci flow we have
\begin{equation*}
\frac{\partial R}{\partial t}=\Delta R+2|Ric|^{2},\,\,\,\,\,\,2\nabla^{i}R_{il}-\nabla_{l}R=0.
\end{equation*}
Therefore we obtain
\begin{equation*}
\mathcal{I}(S,X)=0,\,\,\,\mathcal{D}(S)=0,\,\,\,\mathcal{E}(S,X)=0,\,\,\,\mathcal{H}(S,X)=\frac{\partial R}{\partial t}+\frac{R}{t}-2\nabla_{i}RX^{i}+2R^{ij}X_{i}X_{j}.
\end{equation*}
Notice that for any vector field $X=X^{i}\frac{\partial }{\partial x^{i}}$ on $M$, if $g(t)$ be complete solution to the Ricci flow with bounded curvature  and nonnegative  curvature operator then from \cite{H1} we have  $\mathcal{H}(S,X)\geq 0$, that is  $g(t)$ has weakly positive curvature operator. Hence, the assumption in Theorems \ref{t1}, \ref{t2} and Corollary \ref{c1} hold.
\subsection{List's extended Ricci flow }
Extended  Ricci flow defined by List in  \cite{BL} as follows
\begin{equation*}
\begin{cases}
\frac{\partial }{\partial t}g_{ij}=-2R_{ij}+4\nabla_{i}f\nabla_{j}f,&\\
\frac{\partial f}{\partial t}=\Delta f,&(g(0),f(0))=(g_{0},f_{0}),
\end{cases}
\end{equation*}
where $f:M\to \mathbb{R}$ is a smooth function. In this case, $S_{ij}=R_{ij}-2\nabla_{i}f\nabla_{j}f$ and $S=R-2|\nabla f|^{2}$. Along the extended Ricci flow we have
\begin{equation*}
\frac{\partial S}{\partial t}=\Delta S+2|Ric|^{2}+4|\Delta f|^{2},\,\,\,\,\,\,2\nabla^{i}S_{il}-\nabla_{l}S+4\Delta f\nabla_{l} f=0.
\end{equation*}
Therfore we obtain
\begin{equation*}
\mathcal{I}(S,X)=2(\nabla_{X}f)^{2}\geq0,\,\,\,\mathcal{D}(S)=4|\Delta f|^{2},\,\,\,\mathcal{E}(S,X)=4|\Delta f-\nabla_{X}f|^{2}\geq0.
\end{equation*}
\subsection{M\"{u}ller  coupled with harmonic map flow}
Let $(N,h)$ be a fixed Riemannian manifold. The  harmonic-Ricci flow on $M$ introduced by M\"{u}ller in \cite{M1} as follows
\begin{equation*}
\begin{cases}
\frac{\partial }{\partial t}g_{ij}=-2R_{ij}+2\alpha(t)\nabla_{i}f\nabla_{j}f,&\\
\frac{\partial f}{\partial t}=\tau_{g}f,&,(g(0),f(0))=(g_{0},f_{0})
\end{cases}
\end{equation*}
where $\tau_{g}f$ is the tension  field of the map  $f:M\to N$ with respect to the metric $g(t)$ and $\alpha(t)$ is  positive non-increasing real function respect to $t$.  In this case, $S_{ij}=R_{ij}-\alpha(t)\nabla_{i}f\nabla_{j}f$ and $S=R-\alpha(t)|\nabla f|^{2}$. Along this  flow we have
\begin{equation*}
\frac{\partial S}{\partial t}=\Delta S+2|Ric|^{2}+2\alpha(t)|\tau_{g} f|^{2}-(\frac{\partial \alpha(t)}{\partial t})|\nabla f|^{2},\,\,\,\,\,\,2\nabla^{i}S_{il}-\nabla_{l}S+2\alpha(t)\tau_{g} f\nabla_{l} f=0.
\end{equation*}
Therfore we obtain
\begin{equation*}
\mathcal{I}(S,X)=\alpha(t)\nabla^{i}f\nabla^{j}f X_{i}X_{j}=\alpha(t)(\nabla_{X}f)^{2}\geq0,\,\,\,\mathcal{D}(S)=2\alpha(t)|\tau_{g} f|^{2}-(\frac{\partial \alpha(t)}{\partial t})|\nabla f|^{2}.
\end{equation*}
and
\begin{equation*}
\mathcal{E}(S,X)=2\alpha(t)|\tau_{g} f-\nabla_{X}f|^{2}-(\frac{\partial \alpha(t)}{\partial t})|\nabla f|^{2}.
\end{equation*}
Thus $\mathcal{E}(S,X)\geq 0$ is holds if $\alpha(t)\geq 0$ and  $\alpha(t)\geq 0$  be an non-increasing function. Notice, to the best our knowledge, it is still unknown wether $\mathcal{H}(S,X)\geq 0$  is preserved by the under harmonic-Ricci flow in particular case extended Ricci flow under  suitable assumptions.
\section{Proofs of the results}
In this section, we suppose that $u$ is smooth positive solution  to equation (\ref{6}) and $v=\frac{p}{p-1}u^{p-1}$. In the order to prove the main results, we need the following  lemmas and proposition.
\begin{lemma}
Let $(M, g(t))$ be a complete solution to the geometric flow (\ref{1}) in some time interval $[0,T]$. Suppose that $v$ is  a positive solution of (\ref{7}),

 \begin{equation}\label{8}
\mathcal{L}=\frac{\partial }{\partial t} -(p-1)v\Delta
\end{equation}
and
\begin{eqnarray}\label{9}
F&=&\frac{|\nabla v|^{2}}{v}-b\frac{v_{t}}{v}+(1-b)\frac{S}{v}-\frac{d}{t}\\\nonumber
&=&-b(p-1)\Delta v+(1-b)\frac{|\nabla v|^{2}}{v}-b(p-1)S+(1-b)\frac{S}{v}-\frac{d}{t}.
\end{eqnarray}
then for any constants $b,d$ we have
\begin{eqnarray}\nonumber
\mathcal{L}(F)&=&2p\nabla_{i}F\nabla_{i}v-[\frac{1-b}{v}+p-1]\big(\frac{\partial S}{\partial t}-2\nabla_{i}S\nabla^{i}v+2S^{ij}\nabla_{i}v\nabla_{j}v\big)\\\nonumber
&&-2(p-1)(R^{ij}-S^{ij})\nabla_{i}v\nabla_{j}v-2(p-1)|\nabla^{2}v+\frac{b}{2}S_{ij}|^{2}\\\label{l1}&&+\frac{(b-2)^{2}}{2}(p-1)|S_{ij}|^{2}+(p-1)(1-b)\mathcal{D}(S)\\\nonumber
&&-\frac{1}{b}F^{2}-[(p-1)S-\frac{2(1-b)}{b}\frac{S}{v}+\frac{2d}{bt}]F-\frac{(1-b)^{2}}{b}\frac{S^{2}}{v^{2}}+(1-b)\frac{|\nabla v|^{2}}{v^{2}}S\\\nonumber
&&+\frac{1-b}{b}\frac{|\nabla v|^{4}}{v^{2}}-\frac{d^{2}}{bt^{2}}+\frac{2d}{bt}(1-b)\frac{|\nabla v|^{2}}{v}-d(p-1)\frac{S}{t}+2\frac{1-b}{b}\frac{d}{t}\frac{S}{v}\\\nonumber&&
-2\frac{1-b}{b}\frac{d}{t}+\frac{d}{t^{2}}-b(p-1)(2\nabla^{i}S_{il}-\nabla_{l}S)\nabla^{l}v.
\end{eqnarray}
\end{lemma}
\begin{proof}
First of all, we have the following evolution equations, under the flow (\ref{1}),
\begin{eqnarray}\label{10}
\frac{\partial}{\partial t}(\Delta v)&=&2S^{ij}\nabla_{i}\nabla_{j}v+\Delta(v_{t})-g^{ij}\frac{\partial}{\partial t}(\Gamma_{ij}^{k})\nabla_{k}v\\\label{11}
\frac{\partial}{\partial t}|\nabla v|^{2}&=&2S^{ij}\nabla_{i}v\nabla_{j}v+2\nabla^{i}v_{t}\nabla_{i}v\\\label{12}
g^{ij}\frac{\partial}{\partial t}\Gamma_{ij}^{k}&=&-g^{kl}(2\nabla^{i}S_{il}-\nabla_{l}S).
\end{eqnarray}
Then  from (\ref{5}), (\ref{10}) and (\ref{12}) we get
\begin{eqnarray}\nonumber
\frac{\partial}{\partial t}(\Delta v)&=&2S^{ij}\nabla_{i}\nabla_{j}v+(p-1)v\Delta^{2}v+(p-1)(\nabla v)^{2}+2(p-1)\nabla_{i}(\Delta v)\nabla^{i}v\\\label{13}
&&+\Delta |\nabla v|^{2}+(p-1)\Delta (Sv)+(2\nabla^{i}S_{il}-\nabla_{l}S)\nabla^{l}v.
\end{eqnarray}
Using the Bochner- Weitzebnb\"{o}ck formula
\begin{equation*}
\frac{1}{2}\Delta |\nabla v|^{2}=\nabla_{i}(\Delta v)\nabla^{i}v+|\nabla^{2} v|^{2}+R^{ij}\nabla_{i}v\nabla_{j}v,
\end{equation*}
we obtain
\begin{eqnarray}\nonumber
\mathcal{L}(\Delta v)&=&2p\nabla_{i}(\Delta v)\nabla^{i}v+2S^{ij}\nabla_{i}\nabla_{j}v+(p-1)(\nabla v)^{2}+2|\nabla^{2} v|^{2}+2R^{ij}\nabla_{i}v\nabla_{j}v\\\nonumber
&&+(p-1)v\Delta S+2(p-1)\nabla_{i}S\nabla^{i}v+(p-1)S\Delta v\\\label{14}
&&+(2\nabla^{i}S_{il}-\nabla_{l}S)\nabla^{l}v
\end{eqnarray}
On the other hand, again (\ref{5}) results that
\begin{eqnarray}\nonumber
\mathcal{L}( |\nabla v|^{2})&=&2S^{ij}\nabla_{i}v\nabla_{j}v+2(p-1)|\nabla v|^{2}\Delta v+2\nabla_{i}|\nabla v|^{2}\nabla^{i}v+2(p-1)v\nabla_{i}S\nabla^{i}v\\\label{15}
&&+2(p-1)S|\nabla v|^{2}-2(p-1)v|\nabla^{2} v|^{2}-2(p-1)vR^{ij}\nabla_{i}v\nabla_{j}v,
\end{eqnarray}
it follows that
\begin{eqnarray}\nonumber
\mathcal{L}( \frac{|\nabla v|^{2}}{v})&=&\frac{1}{v}\mathcal{L}( |\nabla v|^{2})-\frac{|\nabla v|^{2}}{v^{2}}\mathcal{L}(v)+2(p-1)\nabla_{i}(\frac{|\nabla v|^{2}}{v})\nabla^{i}v\\\nonumber
&=&2p\nabla_{i}(\frac{|\nabla v|^{2}}{v})\nabla^{i}v+\frac{2}{v}S^{ij}\nabla_{i}v\nabla_{j}v+2(p-1)\frac{|\nabla v|^{2}}{v}\Delta v+\frac{|\nabla v|^{4}}{v^{2}}\\\nonumber
&&+2(p-1)\nabla_{i}S\nabla^{i}v+(p-1)\frac{|\nabla v|^{2}}{v}S-2(p-1)|\nabla^{2} v|^{2}\\\label{16}
&&-2(p-1)R^{ij}\nabla_{i}v\nabla_{j}v.
\end{eqnarray}
Also, we obtain
\begin{equation}\label{17}
\mathcal{L}(\frac{S}{v})=2p\nabla_{i}\frac{S}{v}\nabla^{i}v+\frac{|\nabla v|^{2}}{v^{2}}S-\frac{2}{v}\nabla_{i}S\nabla^{i}v+\frac{1}{v}\frac{\partial S}{\partial t}-(p-1)\frac{S^{2}}{v}-(p-1)\Delta S.
\end{equation}
From  (\ref{9}), (\ref{14}), (\ref{16}) and (\ref{17}) we get
\begin{eqnarray}\nonumber
\mathcal{L}(F)&=&(1-b)\mathcal{L}( \frac{|\nabla v|^{2}}{v})-b(p-1)\mathcal{L}(\Delta v)-b(p-1)\mathcal{L}(S)+(1-b)\mathcal{L}(\frac{S}{v})-\mathcal{L}(\frac{d}{t})\\\nonumber
&=&2p\nabla_{i}F\nabla_{i}v+\frac{1-b}{v}\big(\frac{\partial S}{\partial t}-2\nabla_{i}S\nabla^{i}v+2S^{ij}\nabla_{i}v\nabla_{j}v\big)-2(p-1)|\nabla ^{2}v|^{2}\\
&&-(p-1)\big(b\frac{\partial S}{\partial t}+(1-b)\Delta S-2\nabla_{i}S\nabla^{i}v+2R^{ij}\nabla_{i}v\nabla_{j}v \big)\\\nonumber
&&-2b(p-1)S^{ij}\nabla_{i}\nabla_{j}v-b(p-1)^{2}(\Delta v)^{2}+2(1-b)(p-1)\frac{|\nabla v|^{2}}{v}\Delta v\\\nonumber
&&-b(p-1)^{2}S\Delta v+(1-b)(p-1)\frac{|\nabla v|^{2}}{v}S+(1-b)\frac{|\nabla v|^{4}}{v^{2}}+(1-b)\frac{|\nabla v|^{2}}{v^{2}}S\\\nonumber
&&-(1-b)(p-1)
\frac{S^{2}}{v}-b(p-1)(2\nabla^{i}S_{il}-\nabla_{l}S)\nabla^{l}v+\frac{d}{t^{2}}.
\end{eqnarray}
Since $\Delta S=\frac{\partial S}{\partial t}-2|S_{ij}|^{2}-\mathcal{D}(S)$ and
\begin{eqnarray*}
&&-b(p-1)^{2}(\Delta v)^{2}+2(1-b)(p-1)\frac{|\nabla v|^{2}}{v}\Delta v-b(p-1)^{2}S\Delta v\\&&+(1-b)(p-1)\frac{|\nabla v|^{2}}{v}S
+(1-b)\frac{|\nabla v|^{4}}{v^{2}}+(1-b)\frac{|\nabla v|^{2}}{v^{2}}S-(1-b)(p-1)\frac{S^{2}}{v}\\
&=&-\frac{1}{b}\big( -F+(1-b)\frac{|\nabla v|^{2}}{v}-b(p-1)S+(1-b)\frac{S}{v}-\frac{d}{t}\big)^{2}\\
&&-2\frac{1-b}{b}(p-1)\big( F-(1-b)\frac{|\nabla v|^{2}}{v}+b(p-1)S-(1-b)\frac{S}{v}+\frac{d}{t}\big)\\
&&+(p-1)S\big( F-(1-b)\frac{|\nabla v|^{2}}{v}+b(p-1)S-(1-b)\frac{S}{v}+\frac{d}{t}\big)\\
&&+(1-b)(p-1)\frac{|\nabla v|^{2}}{v}S
+(1-b)\frac{|\nabla v|^{4}}{v^{2}}+(1-b)\frac{|\nabla v|^{2}}{v^{2}}S-(1-b)(p-1)\frac{S^{2}}{v}\\
&=&-\frac{1}{b}F^{2}-[(p-1)S-\frac{2(1-b)}{b}\frac{S}{v}+\frac{2d}{bt}]F-\frac{(1-b)^{2}}{b}\frac{S^{2}}{v^{2}}+(1-b)\frac{|\nabla v|^{2}}{v^{2}}S\\
&&+\frac{1-b}{b}\frac{|\nabla v|^{4}}{v^{2}}-\frac{d^{2}}{bt^{2}}+\frac{2d}{bt}(1-b)\frac{|\nabla v|^{2}}{v}-d(p-1)\frac{S}{t}+2\frac{1-b}{b}\frac{d}{t}\frac{S}{v}-2\frac{1-b}{b}\frac{d}{t},
\end{eqnarray*}
we have
\begin{eqnarray}\nonumber
\mathcal{L}(F)&=&2p\nabla_{i}F\nabla_{i}v+\frac{1-b}{v}\big(\frac{\partial S}{\partial t}-2\nabla_{i}S\nabla^{i}v+2S^{ij}\nabla_{i}v\nabla_{j}v\big)\\\nonumber
&&-(p-1)\big(\frac{\partial S}{\partial t}-2\nabla_{i}S\nabla^{i}v+2R^{ij}\nabla_{i}v\nabla_{j}v \big)+2(p-1)(1-b)|S_{ij}|^{2}\\\label{18}&&+(p-1)(1-b)\mathcal{D}(S)-2(p-1)|\nabla ^{2}v|^{2}-2b(p-1)S^{ij}\nabla_{i}\nabla_{j}v\\\nonumber
&&-\frac{1}{b}F^{2}-[(p-1)S-\frac{2(1-b)}{b}\frac{S}{v}+\frac{2d}{bt}]F-\frac{(1-b)^{2}}{b}\frac{S^{2}}{v^{2}}+(1-b)\frac{|\nabla v|^{2}}{v^{2}}S\\\nonumber
&&+\frac{1-b}{b}\frac{|\nabla v|^{4}}{v^{2}}-\frac{d^{2}}{bt^{2}}+\frac{2d}{bt}(1-b)\frac{|\nabla v|^{2}}{v}-d(p-1)\frac{S}{t}+2\frac{1-b}{b}\frac{d}{t}\frac{S}{v}\\\nonumber&&
-2\frac{1-b}{b}\frac{d}{t}+\frac{d}{t^{2}}-b(p-1)(2\nabla^{i}S_{il}-\nabla_{l}S)\nabla^{l}v.
\end{eqnarray}
Evolution  equation (\ref{18}) results that (\ref{l1}).
\end{proof}
\begin{definition}
Suppose that $g(t)$ evolves by (\ref{1}). Let $S$ be the trace of $S_{ij}$ and $X=X^{i}\frac{\partial }{\partial x^{i}}$ be  a vector field on $M$. We define
\begin{equation*}
\mathcal{E}_{b}(S,X)=(b-1)\mathcal{D}(S)+2\mathcal{I}(S,X)+b(2\nabla^{i}S_{ij}-\nabla_{j}S)X^{j}
\end{equation*}
where $b$ is a constant.
\end{definition}
\begin{proposition}\label{p1}
Let $g(t)$, $t\in[0,T)$ be a solution to the geometric flow (\ref{1}) on a closed   Riemannian $n$-manifold $M$ satisfying
\begin{equation}
\mathcal{E}_{b}(S,X)\geq0,\,\,\,\mathcal{H}(S,X)\geq0,\,\,\,Ric\geq -(n-1)k_{1},\,\,\,\,-k_{2}g\leq S_{ij}\leq k_{3}g,\,\,\, S\geq 0
\end{equation}
for all vector fields $X$ and  all time $t\in[0,T)$. Suppose $u$ is  a smooth positive  solution  to eqaution (\ref{6}) with $p>1$  and $v=\frac{p}{p-1}u^{p-1}$. Then
for any $b\in[2,\infty)$ and $d\geq b$, on the geodesic ball $\mathcal{Q}_{\rho,T}$, we have
\begin{equation}\label{pp1}
\frac{|\nabla v|^{2}}{v}-b\frac{v_{t}}{v}-(b-1)\frac{S}{v}-\frac{d}{t}\leq b\alpha(\frac{E_{4}v_{\max}}{\rho^{2}}+E_{5})+E_{6}
\end{equation}
where $\alpha=\frac{bn(p-1)}{2+bn(p-1)}$, $E_{4}=\big(\frac{b^{2}p^{2}n}{4(b-1)}+\frac{\sqrt{k_{1}}\rho}{2}+\frac{9}{4} \big)c_{3}(p-1)$,
$E_{5}=\sqrt{c_{4}}(k_{2}+k_{3})^{2}+\frac{2(b-2)}{b}(k_{2}+k_{3})+1$ and $E_{6}=n(k_{2}+k_{3})(b-2)\sqrt{\frac{b(p-1)\alpha}{2}}$.
\end{proposition}
\begin{proof}
Let $x,x_{0}$ and $d(x,x_{0},t)$ be the geodesic distance  $x$ from $x_{0}$ with respect to the metric  $g(t)$. Choose a smooth cut-off function $\psi(s)$ defined on $[0,+\infty)$ with $\psi(s)=1$ for $0\leq s\leq \frac{1}{2}$, $\psi(s)=0$ for $1\leq s$ and $\psi(s)>0$ for $\frac{1}{2}< s<1 $ such that $-c_{1}\psi^{\frac{1}{2}}\leq \psi'(s)\leq0$, $-c_{2}\leq \psi''(s)\leq c_{2}$ and  $-c_{2}\psi\leq|\psi'|^{2}\leq c_{2}\psi$ for some absolute constants $c_{1},c_{2}>0$. For any fixed point  $x_{0}\in M$  and any positive number $\rho>0$, we define $\phi(x,t)=\psi(\frac{r(x,t)}{2\rho})$ on
\begin{equation}
\mathcal{Q}_{\rho,T}=B(x_{0},2\rho)\times[0,T)\subset M\times[0,+\infty)
\end{equation}
where $B(x_{0},2\rho)$ is a ball of radius $2\rho>0$ centered at $x_{0}$ and $r(x,t)=d(x,x_{0},t)$. Using an argument of Calabi \cite{EC}, we can assume every where  smooth ness of $\phi(x,t)$ with  support in $\mathcal{Q}_{\rho,T}$. By the Laplacian comparison theorem in \cite{TA}, the Laplacian  of the distance function satisfies
\begin{equation}\label{19}
\Delta r(x,t)\leq (n-1)\sqrt{|k_{1}|}\coth (2\sqrt{|k_{1}|}\rho),\,\,\,\,\,\forall x\in M,\,\,\,d(x,x_{0})\geq2\rho.
\end{equation}
From the definition of $\phi$ and direct calculation  shows that
\begin{equation*}
\frac{|\nabla \phi|^{2}}{\phi}=\frac{|\psi'|^{2}|\nabla r|^{2}}{4\rho^{2}}\leq \frac{c_{1}}{\rho^{2}},
\end{equation*}
and
\begin{equation*}
\Delta \phi=\frac{\psi' \Delta r}{2\rho}+\frac{\psi''|\nabla r|^{2}}{4\rho^{2}}
\geq -\frac{c_{1}}{2\rho}(n-1)\sqrt{|k_{1}|}\coth (2\sqrt{|k_{1}|}\rho)-\frac{c_{1}}{4\rho^{2}}\geq -\frac{c_{1}\sqrt{|k_{1}|}}{2\rho}-\frac{c_{1}}{4\rho^{2}}.
\end{equation*}
On the other hand, since along the geometric flow (\ref{1}), for  a fixed smooth path $\gamma:[a,b]\to M$ whose length at time  $t$ is given by
$d(\gamma)=\int_{a}^{b}|\gamma'(s)|_{g(t)}ds$,  where $s$ is the arc length  along the path, we have
\begin{equation*}
\frac{\partial d(\gamma)}{\partial t}=-\frac{1}{2}\int_{a}^{b}|\gamma'(s)|_{g(t)}^{-1}S_{ij}(X,X)ds
\end{equation*}
where $X$ is the unit tangent vector  to the path $\gamma$. $-k_{2}g \geq S_{ij}\leq k_{3}g$ results that $-(k_{2}+k_{3})g\leq S_{ij}\leq(k_{2}+k_{3})g $, then
\begin{equation*}
\sup_{M}|S_{ij}|^{2}\leq n(k_{2}+k_{3})^{2}.
\end{equation*}
Now, we get
\begin{equation*}
\frac{\partial \phi}{\partial t}=\frac{\psi'}{2\rho}\frac{\partial r}{\partial t} =\frac{\psi'}{2\rho}\int_{\gamma}S_{ij}(X,X)ds\leq \sqrt{c_{2}}(k_{2}+k_{3})^{2}.
\end{equation*}
Suppose that $t\phi F$ achieves its positive maximum value  at $(v_{0},t_{0})$. Then at $(x_{0},t_{0})$, we have
\begin{equation*}
\nabla(t\phi F)(x_{0},t_{0})=0,\,\,\,\,\frac{\partial }{\partial t}(t\phi F)(x_{0},t_{0})\geq0,\,\,\,\,\,\mathcal{L}(t\phi F)(x_{0},t_{0})\geq 0.
\end{equation*}
Suppose that
\begin{equation*}
y=\frac{|\nabla v|^{2}}{v}+\frac{S}{v},\,\,\,\tilde{y}=t\phi y, \,\,\,\,z=\frac{v_{t}}{v}+\frac{S}{v}+\frac{d}{bt},\,\,\,\tilde{z}=t\phi z
\end{equation*}
then $F=y-bz$, $t\phi F=\tilde{y}-b\tilde{z}$ and
\begin{eqnarray}\nonumber
\mathcal{L}(F)&=&2p\nabla_{i}F\nabla_{i}v-[\frac{1-b}{v}+p-1]\big(\frac{\partial S}{\partial t}-2\nabla_{i}S\nabla^{i}v+2S^{ij}\nabla_{i}v\nabla_{j}v\big)\\\nonumber
&&-2(p-1)(R^{ij}-S^{ij})\nabla_{i}v\nabla_{j}v-2(p-1)|\nabla^{2}v+\frac{b}{2}S_{ij}|^{2}\\\label{20}&&+\frac{(b-2)^{2}}{2}(p-1)|S_{ij}|^{2}+(p-1)(1-b)\mathcal{D}(S)-\frac{1}{b}F^{2}\\\nonumber
&&-[(p-1)S-\frac{2(1-b)}{b}\frac{S}{v}+\frac{2d}{bt}]F-\frac{b-1}{b}y^{2}-\frac{(b-1)(b-2)}{b}y\frac{S}{v}
\\\nonumber
&&-\frac{d^{2}}{bt^{2}}+\frac{2d}{bt}(1-b)\frac{|\nabla v|^{2}}{v}-d(p-1)\frac{S}{t}+2\frac{1-b}{b}\frac{d}{t}\frac{S}{v}-2\frac{1-b}{b}\frac{d}{t}\\\nonumber&&
+\frac{d}{t^{2}}-b(p-1)(2\nabla^{i}S_{il}-\nabla_{l}S)\nabla^{l}v.
\end{eqnarray}
Therefore
\begin{eqnarray}\nonumber
t\phi\mathcal{L}(t\phi F)&=&t\phi^{2}+t^{2}\phi \phi_{t} F+ t^{2} \phi^{2} \mathcal{L}(F)-(p-1)t^{2}\phi v F\Delta \phi \\\nonumber&&-2t^{2}(p-1)\phi v \nabla_{i}\phi \nabla^{i}F\\\nonumber
&=&\phi(\tilde{y}-b\tilde{z})+t\phi_{t}(\tilde{y}-b\tilde{z})-(p-1)tv\Delta \phi (\tilde{y}-b\tilde{z})\\\nonumber
&&
-2t^{2}(p-1)\phi v \nabla_{i}\phi \nabla^{i}F+2pt^{2}\phi^{2}\nabla_{i}F\nabla_{i}v\\\nonumber
&&-t^{2}\phi^{2}[\frac{1-b}{v}+p-1]\mathcal{H}(S,\nabla v)-2(p-1)t^{2}\phi^{2}|\nabla^{2}v+\frac{b}{2}S_{ij}|^{2}\\\label{21}&&+\frac{(b-2)^{2}}{2}(p-1)t^{2}\phi^{2}|S_{ij}|^{2}-\frac{1}{b}(\tilde{y}-b\tilde{z})^{2}\\\nonumber
&&\\\nonumber&&-[(p-1)S-\frac{2(1-b)}{b}\frac{S}{v}+\frac{2d}{bt}]t^{2}\phi^{2}F-\frac{b-1}{b}\tilde{y}^{2}
\\\nonumber
&&-\frac{(b-1)(b-2)}{b}t\phi\frac{S}{v}\tilde{y}+\frac{2d}{b}(1-b)t\phi^{2}\frac{|\nabla v|^{2}}{v}\\\nonumber&&
-d(p-1)t\phi^{2}S+2\frac{1-b}{b}d t\phi^{2}\frac{S}{v}-2\frac{1-b}{b}d t\phi^{2}\\\nonumber
&&
-t^{2}\phi^{2}(p-1)\mathcal{E}_{b}(S,\nabla v)-\frac{d}{b}(d-b)\phi^{2}.
\end{eqnarray}
On the other hand, we have
\begin{eqnarray*}
t_{0}^{2}\phi^{2}\nabla_{i}F\nabla^{i}v=-t_{0}^{2}\phi F\nabla_{i}\phi\nabla^{i}v\leq t_{0}^{2}\phi F|\nabla_{i}\phi| |\nabla^{i}v|\leq \frac{\sqrt{c_{1}}}{\rho}\tilde{y}^{\frac{1}{2}}(t_{0}v)^{\frac{1}{2}}(\tilde{y}-b\tilde{z}),
\\
-(p-1)t_{0}v\Delta \phi (\tilde{y}-b\tilde{z})\leq (p-1)t_{0}v(\frac{c_{1}\sqrt{k_{1}}}{2\rho}+\frac{c_{1}}{4\rho^{2}})(\tilde{y}-b\tilde{z}),
\\
-2(p-1)t_{0}^{2}v\phi\nabla_{i}\phi\nabla^{i}F=2(p-1)t_{0}^{2}v|\nabla \phi|^{2}F\leq 2(p-1)t_{0}v\frac{c_{1}}{\rho^{2}}(\tilde{y}-b\tilde{z}),
\end{eqnarray*}
and
\begin{eqnarray}\nonumber
&&
-2(p-1)t_{0}^{2}\phi^{2}|\nabla^{2}v+\frac{b}{2}S_{ij}|^{2}
\leq -\frac{2(p-1)t_{0}^{2}\phi^{2}}{n}(\Delta v+\frac{b}{2}S)^{2}\\
&=&-\frac{2(p-1)t_{0}^{2}\phi^{2}}{n}\left(-\frac{F}{b(p-1)}-\frac{b-1}{b(p-1)}\frac{|\nabla v|^{2}}{v}+\frac{b-2}{2}S-\frac{b-1}{b(p-1)}\frac{S}{v}\right)^{2}\\\nonumber
&=&-\frac{2}{b^{2}n(p-1)}\left(\tilde{y}-b\tilde{z}+(b-1)\tilde{y}-\frac{b(b-2)}{2}(p-1)t_{0}\phi S\right)^{2}.
\end{eqnarray}
Thus
\begin{eqnarray}\nonumber
0&\leq& t_{0}\phi \mathcal{L}(t_{0}\phi F)\\\nonumber&\leq&
(\tilde{y}-b\tilde{z})+t_{0}\sqrt{c_{2}}(k_{2}+k_{3})^{2}(\tilde{y}-b\tilde{z})+
(p-1)t_{0}v(\frac{c_{1}\sqrt{k_{1}}}{2\rho}+\frac{c_{1}}{4\rho^{2}})(\tilde{y}-b\tilde{z})\\&&
+2(p-1)t_{0}v\frac{c_{1}}{\rho^{2}}(\tilde{y}-b\tilde{z})+2p\frac{\sqrt{c_{1}}}{\rho}\tilde{y}^{\frac{1}{2}}(t_{0}v)^{\frac{1}{2}}(\tilde{y}-b\tilde{z})\\\nonumber &&
-\frac{2}{b^{2}n(p-1)}\left(\tilde{y}-b\tilde{z}+(b-1)\tilde{y}-\frac{b(b-2)}{2}(p-1)t_{0}\phi S\right)^{2}\\\nonumber
&&+\frac{(b-2)^{2}}{2}(p-1)t_{0}^{2}\phi^{2}S^{2}-\frac{1}{b}(\tilde{y}-b\tilde{z})^{2}.
\end{eqnarray}
Notice that $(r+s)^{2}\geq r^{2}+2rs$ results that
\begin{eqnarray}\nonumber
&&
-\frac{2}{b^{2}n(p-1)}\left(\tilde{y}-b\tilde{z}+(b-1)\tilde{y}-\frac{b(b-2)}{2}(p-1)t_{0}\phi S\right)^{2}\\\nonumber
&\leq&-\frac{2}{b^{2}n(p-1)}(\tilde{y}-b\tilde{z})^{2}-\frac{4(b-1)}{b^{2}n(p-1)}\tilde{y}(\tilde{y}-b\tilde{z})+\frac{2(b-2)}{bn}t_{0}\phi S (\tilde{y}-b\tilde{z})
\end{eqnarray}
hence
\begin{eqnarray}\nonumber
0&\leq&
(\tilde{y}-b\tilde{z})+t_{0}\sqrt{c_{2}}(k_{2}+k_{3})^{2}(\tilde{y}-b\tilde{z})+
(p-1)t_{0}v(\frac{c_{1}\sqrt{k_{1}}}{2\rho}+\frac{c_{1}}{4\rho^{2}})(\tilde{y}-b\tilde{z})\\\nonumber&&
+2(p-1)t_{0}v\frac{c_{1}}{\rho^{2}}(\tilde{y}-b\tilde{z})+2p\frac{\sqrt{c_{1}}}{\rho}\tilde{y}^{\frac{1}{2}}(t_{0}v)^{\frac{1}{2}}(\tilde{y}-b\tilde{z})-\frac{1}{b\alpha}(\tilde{y}-b\tilde{z})^{2}\\\nonumber &&
-\frac{4(b-1)}{b^{2}n(p-1)}\tilde{y}(\tilde{y}-b\tilde{z})+\frac{2(b-2)}{bn}t_{0}\phi S (\tilde{y}-b\tilde{z})+\frac{(b-2)^{2}}{2}(p-1)t_{0}^{2}\phi^{2}S^{2}\\\nonumber
&\leq& (\tilde{y}-b\tilde{z})\left[-\frac{4(b-1)}{b^{2}n(p-1)}\tilde{y}+2p\frac{\sqrt{c_{1}}}{\rho}\tilde{y}^{\frac{1}{2}}(t_{0}v)^{\frac{1}{2}}+(\frac{c_{1}\sqrt{k_{1}}}{2\rho}+\frac{c_{1}}{4\rho^{2}}+2\frac{c_{1}}{\rho^{2}})(p-1)t_{0}v \right]\\\nonumber&&
+(\tilde{y}-b\tilde{z})\left[ t_{0}\sqrt{c_{2}}(k_{2}+k_{3})^{2}+\frac{2(b-2)}{b}t_{0} (k_{2}+k_{3})+1 \right]\\\nonumber&&+\frac{(b-2)^{2}}{2}(p-1)t_{0}^{2}n^{2}(k_{2}+k_{3})^{2}-\frac{1}{b\alpha}(\tilde{y}-b\tilde{z})^{2}.
\end{eqnarray}
where $\alpha=\frac{bn(p-1)}{2+bn(p-1)}$. For $a>0$  inequality $-ax^{2}+bx\leq \frac{b^{2}}{4a}$ implies that
\begin{equation}
-\frac{4(b-1)}{b^{2}n(p-1)}\tilde{y}+2p\frac{\sqrt{c_{1}}}{\rho}\tilde{y}^{\frac{1}{2}}(t_{0}v)^{\frac{1}{2}}\leq
\frac{b^{2}p^{2}nc_{1}}{4(b-1)\rho^{2}}(p-1)t_{0}v.
\end{equation}
Therefore
\begin{eqnarray}\nonumber
0&\leq& (\tilde{y}-b\tilde{z})\left[\left(\frac{b^{2}p^{2}nc_{1}}{4(b-1)\rho^{2}} +\frac{c_{1}\sqrt{k_{1}}}{2\rho}+\frac{c_{1}}{4\rho^{2}}+2\frac{c_{1}}{\rho^{2}} \right)(p-1)t_{0}v+t_{0}\sqrt{c_{2}}(k_{2}+k_{3})^{2}\right.\\&&\left.+\frac{2(b-2)}{b}t_{0} (k_{2}+k_{3})+1 \right]+\frac{(b-2)^{2}}{2}(p-1)t_{0}^{2}n^{2}(k_{2}+k_{3})^{2}-\frac{1}{b\alpha}(\tilde{y}-b\tilde{z})^{2}.
\end{eqnarray}
If $0\leq -ax^{2}+bx+c$ for $a,b,c>0$ then $x\leq \frac{b}{a}+\sqrt{\frac{c}{a}}$. Hence
\begin{eqnarray}\nonumber
\tilde{y}-b\tilde{z}&\leq&b\alpha\left[\left(\frac{b^{2}p^{2}nc_{1}}{4(b-1)\rho^{2}} +\frac{c_{1}\sqrt{k_{1}}}{2\rho}+\frac{c_{1}}{4\rho^{2}}+2\frac{c_{1}}{\rho^{2}} \right)(p-1)t_{0}v+t_{0}\sqrt{c_{2}}(k_{2}+k_{3})^{2}\right.\\&&\left.+\frac{2(b-2)}{b}t_{0} (k_{2}+k_{3})+1 \right]+t_{0}n(k_{2}+k_{3})(b-2)\sqrt{\frac{b(p-1)\alpha}{2}}.
\end{eqnarray}
If $d(x,x_{0},\tau)<2\rho^{2}$ then $\phi(x,\tau)=1$. Then since $(x_{0},t_{0})$ is the maximum point for $t\phi F$ in $\mathcal{Q}_{\rho,T}$, we have
\begin{equation}
\tau F(x,\tau)=(\tau \phi F)(x,\tau)\leq (t_{0}\phi F)(x_{0},t_{0}).
\end{equation}
For all $x\in M$, such that $d(x,x_{0},\tau)<2\rho^{2}$ and $\tau\in[0,T]$ is arbitrary. Then we have
\begin{eqnarray}\nonumber
F&\leq&b\alpha\left[\left(\frac{b^{2}p^{2}nc_{1}}{4(b-1)\rho^{2}} +\frac{c_{1}\sqrt{k_{1}}}{2\rho}+\frac{c_{1}}{4\rho^{2}}+2\frac{c_{1}}{\rho^{2}} \right)(p-1)v+\sqrt{c_{2}}(k_{2}+k_{3})^{2}\right.\\&&\left.+\frac{2(b-2)}{b} (k_{2}+k_{3})+1 \right]+n(k_{2}+k_{3})(b-2)\sqrt{\frac{b(p-1)\alpha}{2}}.
\end{eqnarray}
This completes the proof.
\end{proof}
{\bf Proof of Theorem  \ref{t1}.}
In Proposition \ref{p1}, suppose that $b=2$. Then inequality (\ref{pp1}) results that ({\ref{ft1}).\\

{\bf Proof of Corollary \ref{c1}.} If $u$ is bounded on  $M\times [0,T]$, then  assume that  $\rho\to \infty$, therefore inequality Theorem \ref{t1} results that (\ref{ft2}).\\

{\bf Proof of Theorem \ref{t2}.} For any curve $\gamma(t),\,\,t\in[t_{1},t_{2}]$, from $\gamma(t_{1})=x_{1}$ to $\gamma(t_{2})=x_{2}$, we have
\begin{equation*}
\log \frac{v(x_{1},t_{1})}{ v(x_{2},t_{2})}=\int_{t_{1}}^{t_{2}}\frac{d}{dt}\log\, v(\gamma(t),t)dt =\int_{t_{1}}^{t_{2}}
\frac{v_{t}}{v}+\frac{\nabla v}{v}\frac{d\gamma}{dt}dt.
\end{equation*}
Since for any $x,y$, inequality $xy\geq -\frac{x^{2}}{2}-\frac{y^{2}}{2}$ results that
\begin{equation*}
\nabla v.\frac{d\gamma}{dt}\geq -\frac{|\nabla v|^{2}}{2}-\frac{1}{2}|\frac{d\gamma}{dt}|^{2}.
\end{equation*}
Hence
\begin{equation*}
\log \frac{v(x_{1},t_{1})}{ v(x_{2},t_{2})}\geq\int_{t_{1}}^{t_{2}}(\frac{v_{t}}{v}-\frac{|\nabla v|^{2}}{2v}-\frac{1}{2v}|\frac{d\gamma}{dt}|^{2})dt.
\end{equation*}
Corollary \ref{c1} implies that
\begin{eqnarray*}
\log \frac{v(x_{1},t_{1})}{ v(x_{2},t_{2})}&\geq&\int_{t_{1}}^{t_{2}}\big(-\frac{n(p-1)}{1+n(p-1)}E_{2}-\frac{S}{2v_{\min}}-\frac{d}{2t}-\frac{1}{2v_{\min}}|\frac{d\gamma}{dt}|^{2} \big)dt\\
&=&-\frac{n(p-1)}{1+n(p-1)}E_{2}(t_{2}-t_{1})-(\frac{t_{2}}{t_{1}})^{\frac{d}{2}}-\frac{1}{2v_{\min}}\int_{t_{1}}^{t_{2}}\big(S+|\frac{d\gamma}{dt}|^{2} \big)dt.
\end{eqnarray*}
By exponentiating we arrive at (\ref{ftt1}).


\end{document}